\documentclass[12pt]{article}

\usepackage{latexsym,amssymb,amsmath,bm}

\newcommand{\C}{\mathbb C}

\newcommand{\N}{\mathbb N}
\newcommand{\F}{\mathbb F}
\newcommand{\sma}{\left(\begin{array}}
\newcommand{\fma}{\end{array}\right)}

\newtheorem{lem}{Lemma}

\newtheorem{co}[lem]{Corollary}
\newtheorem{thm}[lem]{Theorem}
\newtheorem{prop}[lem]{Proposition}

\newenvironment{proof}{\textbf{Proof.}}{\newline\hspace*{\fill}{$\Box$}\\}

\begin{document}
\title{Mapping class groups are not linear in positive characteristic}
\author{J.\,O.\,Button}

\newcommand{\Address}{{
  \bigskip
  \footnotesize

\textsc{Selwyn College, University of Cambridge,
Cambridge CB3 9DQ, UK}\par\nopagebreak
  \textit{E-mail address}: \texttt{j.o.button@dpmms.cam.ac.uk}
}}

\date{}
\maketitle
\begin{abstract}
For $\Sigma$ an orientable surface of finite topological type having
genus at least 3 (possibly closed or possibly with any
number of punctures or boundary components), we show that the mapping
class group $Mod(\Sigma)$ has no faithful linear representation 
in any dimension over any field of positive characteristic. 
\end{abstract}

\section{Introduction}

A common question to ask of a given infinite
finitely generated group is whether
it is linear. For instance consider the braid groups $B_n$, the automorphism
group $Aut(F_n)$ of the free group $F_n$ and the mapping class group
$Mod(\Sigma_g)$ of the closed orientable surface $\Sigma_g$ with genus $g$.
Linearity in the first case was open for a while but is now known to hold
by \cite{big}, \cite{kra}. For $n\geq 3$ \cite{fp} showed that $Aut(F_n)$
is not linear, whereas for $g\geq 3$ the third case is open. However
whereas the definition of linearity is that a group embeds in $GL(d,\F)$
for some $d\in\N$ and some field $\F$, in practice one tends to
concentrate on the case where $\F=\C$. In fact a finitely generated
group embeds in $\C$ if and only if it embeds in some field of characteristic
zero, so it is enough to restrict to this case if only characteristic zero
representations are being considered.

However we can still ask about faithful linear representations in positive
characteristic. For instance in the three examples above, it is unknown
for $n\geq 4$ if the braid group $B_n$ admits a faithful linear
representation in any dimension over any field of positive characteristic.
For $Aut(F_n)$ with $n\geq 3$, the proof in \cite{fp} applies to any
field, not just the characteristic zero case, so that there are also no
faithful representations in positive characteristic. As for mapping
class groups, we show here 
that there are no faithful linear representations of $Mod(\Sigma_g)$  
in any dimension over any field of positive characteristic when
$\Sigma_g$ is an orientable surface of finite topological type having
genus $g$ at least 3 (which might be closed or might have any
number of punctures or boundary components). 

The idea comes from considering the analogy between a finitely generated
group being linear in positive characteristic and having a ``nice'' geometric
action, as we did in \cite{me} when showing that Gersten's free by cyclic
group has no faithful linear representation in any positive characteristic.
On looking more closely to see which definition of ``nice'' aligns most
closely with linearity in positive characteristic, we were struck by the
similarities between that and the notion of a finitely generated group
acting properly and semisimply (more so than properly and cocompactly)
on a complete CAT(0) space. In \cite{bri} Bridson shows that for all the
surfaces $\Sigma_g$ mentioned above, the mapping class group $Mod(\Sigma_g)$
does not admit such an action. This result is first credited to \cite{kaplb}
but the proof in \cite{bri} consists of finding an obstruction to the
existence of such
an action by any one of these groups. This obstruction
involves taking an element of infinite order and its
centraliser in said group, then applying a condition on the abelianisation
of this centraliser. Here we show that this condition holds verbatim for
groups which are linear in positive characteristic, thus obtaining the
same obstruction.

We leave open the question of whether the mapping class group of the
closed orientable surface of genus 2 is linear in positive characteristic,
but we note that it was shown in \cite{bbud} and \cite{kork}
using the braid group results that this group is linear in characteristic
zero anyway.

\section{Proof}
The following is the crucial point which distinguishes our treatment of
linear groups in positive characteristic from the classical case.
\begin{prop} \label{diag}
If $\F$ is an algebraically closed field of positive characteristic
and $d\in\N$ then there exists $K\in \N$ such that for all elements
$g\in GL(d,\F)$ the matrix $g^K$ is diagonalisable.
\end{prop}
\begin{proof}
If $\F$ has characteristic $p$ then we take $K$ to be any power of $p$
which is at least $d$.

We put $g$ into Jordan normal form, or indeed any form where the matrix
splits up into blocks corresponding to the generalised eigenspaces of
$g$ and such that we are upper triangular in each block. Then on taking
the eigenvalue $\lambda\in\F$ of $g$, the block of $g$ corresponding
to $\lambda$ will be of the form $\lambda I+N$ where $N$ is upper
triangular with all zeros on the diagonal so that
\[(\lambda I+N)^K=\lambda^KI+\binom{K}{1}\lambda^{K-1}N+\ldots
+\binom{K}{K-1}\lambda N^{K-1}+N^K.\]
But $N^K=0$ because $K\geq d$ and $\binom{K}{i}\equiv 0$ modulo $p$
for $0<i<K$ as $K$ is a power of $p$. Thus in this block we have that
$g^K$ is equal to $\lambda^KI$. But we can do this in each block, making
$g^K$ a diagonal matrix.
\end{proof}

As for the mapping class group $Mod(\Sigma)$ of the surface $\Sigma$,
we have:
\begin{prop} \label{bri} (\cite{bri} Proposition 4.2)\\
If $\Sigma$ is an orientable surface of finite type having genus at
least 3 (with any number of boundary components and punctures) and if $T$
is the Dehn twist about any simple closed curve in $\Sigma$ then the
abelianisation of the centraliser in $Mod(\Sigma)$ of $T$ is finite.
\end{prop}

This is in contrast to:
\begin{thm} \label{but}
Suppose that $G$ is a linear group over a field of positive
characteristic and $C_G(g)$ is the centraliser in $G$ of the infinite order
element $g$. Then the image of $g$ in the abelianisation of $C_G(g)$ also
has infinite order.
\end{thm}
\begin{proof}
As the abelianisation is the universal abelian quotient of a group, it is
enough to find some homomorphism of $C_G(g)$ to an abelian group
where $g$ maps to an element of infinite order, so we use the
determinant.

We first replace our field by its algebraic closure. Then
Proposition \ref{diag} tells us that we have the diagonalisable element
$g^K$, whereupon showing that $g^K$ has infinite order in the
abelianisation of $C_G(g)$ (which could of course be smaller than the
centraliser in $G$ of $g^K$) will establish the same for $g$.

Take a basis so that $g^K$ is actually diagonal and group together repeated
eigenvalues, so that we have
\[g^K=\sma{ccc}\boxed{\lambda_1I_{d_1}}& &0\\
&\ddots&\\
0& &\boxed{\lambda_kI_{d_k}}\fma.
\]
This means that any element in $C_G(g^K)$, and thus also in $C_G(g)$, is of the
form
\[\sma{ccc}\boxed{A_1}& &0\\
&\ddots&\\
0& &\boxed{A_k}\fma
\]
with the same block structure.

Consequently we have as homomorphisms from $C_G(g)$ to the multiplicative
abelian group $(\F^*,\times)$ not just the determinant itself but also
the ``subdeterminant'' functions
$\mbox{det}_1,\ldots ,\mbox{det}_k$. Here for $h\in C_G(g)$ we define
$\mbox{det}_i(h)$ as the determinant of the $i$th block of $h$ when
expressed with respect to our basis above which diagonalises $g^K$, and
this is indeed a homomorphism.

Now it could be that $\mbox{det}_i(g^K)$ has finite order, which implies
that $\lambda_i^{d_i}$ and thus also $\lambda_i$ has finite multiplicative
order in $\F^*$. However if this is true for all $i=1,2,\ldots ,k$ then
$\lambda _1,\ldots ,\lambda_k$ all have finite order. This means that
$g^K$ and so $g$ does too, which is a contradiction. Thus for some $i$ we know
$g^K$ and $g$ map to elements of infinite order in the 
abelian group $(\F^*,\times)$ under the homomorphism $\mbox{det}_i$.
\end{proof}

\begin{co} If $\Sigma$ is an orientable surface of finite type having genus at
least 3 (with any number of boundary components and punctures) then
$Mod(\Sigma)$ is not linear over any field of positive characteristic.
\end{co}
\begin{proof} We can combine Proposition \ref{bri} and Theorem \ref{but}
to get a contradiction because Dehn twists have infinite order.
\end{proof}

\Address


\begin{thebibliography}{9}

\bibitem{big}  S.\,J.\,Bigelow, 
{\it Braid groups are linear}, 
J. Amer. Math. Soc. {\bf 14} (2001) 471--486.

\bibitem{bbud} S.\,J.\,Bigelow and R.\,D.\,Budney,
{\it The mapping class group of a genus two surface is linear},
Algebr. Geom. Topol. {\bf 1} (2001) 699-–708.

\bibitem{bri} M.\,R.\,Bridson, 
{\it Semisimple actions of mapping class 
groups on CAT(0) spaces}, Geometry of Riemann surfaces, 1--14, London 
Math. Soc. Lecture Note Ser., 368, Cambridge Univ. Press, Cambridge, 2010.

\bibitem{me} J.\,O.\,Button,
{\it Minimal dimension faithful linear representations of common 
finitely presented groups},
\texttt{http://arxiv.org/1610.03712}

\bibitem{fp} E.\,Formanek and C.\,Procesi,
{\it The automorphism group of a free group is not linear},
J. Algebra {\bf 149} (1992) 494--499.

\bibitem{kaplb} M.\,Kapovich and B.\,Leeb,
{\it Actions of discrete groups on nonpositively curved spaces},
Math. Ann. {\bf 306} (1996) 341--352. 

\bibitem{kork} M.\,Korkmaz, 
{\it On the linearity of certain mapping class groups}, 
Turkish J. Math. {\bf 24} (2000) 367-–371.

\bibitem{kra} D.\,Krammer,
{\it Braid groups are linear},
Ann. of Math. {\bf 155} (2002) 131-–156.

\end{thebibliography}
\end{document}